\documentclass[12pt,a4paper]{article}
\usepackage[centertags]{amsmath}
\usepackage{amsfonts}
\usepackage{calc}
\usepackage{tikz}
\usetikzlibrary{decorations.markings}
\tikzstyle{vertex}=[circle, draw, inner sep=0pt, minimum size=6pt]

\usepackage{graphics}
\usepackage{amssymb}
\usepackage{amsthm}
\usepackage{newlfont}
\usepackage{epsfig}
\usepackage{amssymb}
\usepackage{amsmath}
\usepackage{amsthm}
\usepackage{graphicx}
\usepackage{makeidx}
\theoremstyle{plain}
\newtheorem{thm}{Theorem}[section]
\newtheorem{cor}[thm]{Corollary}
\newtheorem{lem}[thm]{Lemma}

\theoremstyle{definition}

\theoremstyle{remark} \tolerance=10000 \hbadness=10000
\def \ni{\noindent}

\author{Jismy Varghese \footnote{Email : kvjismy@gmail.com}\\ School of Computer Science\\
DePaul Institute of Science and Technology\\ Angamaly - 683 573\\ \vspace{0.3cm} Kerala, India.\\
Aparna Lakshmanan S.	\footnote{E-mail : aparnals@cusat.ac.in, aparnaren@gmail.com}\\
Department of Mathematics\\
Cochin University of Science and Technology\\Cochin -
22\\\vspace{0.2cm} Kerala, India.}

\title{\textbf{Perfect Italian Domination Number of Graphs}}

\date{}
\begin{document}
	
	\maketitle
	
	\begin{abstract}
		
		\ni\line(1,0){360}\\ In this paper, an upper bound for the perfect Italian domination number of the cartesian product of any two graphs is obtained and the exact value of this parameter for cartesian product of some special graphs are obtained. We have also proved that for any two positive integers $a$, $b$ there exists a graph $G$ and an induced subgraph $H$ of $G$ such that $ \gamma^{p}_{I}(G) = a $ and  $ \gamma^{p}_{I}(H) = b $. Relationship of the perfect Italian domination number with the Roman domination number and the perfect domination number of a graph $G$ are obtained and the corresponding realization problems are also solved. We have also obtained the perfect Italian domination number of the Mycielskian of a graph in terms of the perfect domination number of the graph. Some open problems related to this parameters are also included.\\
		\ni\line(1,0){360}\\
		\ni {\bf Keywords:} Roman domination number, Perfect domination number, Perfect Italian domination number, Cartesian product, Mycielskian of a graph.  \\
		
		\ni {\bf AMS Subject Classification:} {\bf primary: 05C69, secondary: 05C76}\\
		\ni\line(1,0){360}\\
	\end{abstract}

\section{Introduction}

Let G be a simple graph with vertex set $V(G)$ and edge set $ E(G) $. If there is no ambiguity in the choice of $G$, then we write $V(G)$ and $E(G)$ as $V$ and $E$ respectively. A subset $ S \subseteq V(G)$ of vertices  is called a dominating set if every $ v\in V(G)$ is either an element of $S$ or is adjacent to an element of $S$ \cite{Tha}. The domination number, $\gamma(G) $ is the minimum cardinality of a dominating set of $G$. A dominating set $ S $ is a perfect dominating set if $\rvert N(v) \bigcap S \rvert =1  $ for each $ v \in V-S $, where $ N(v) $ is the collection of all vertices that are adjacent to the vertex $v$. The perfect domination number, $ \gamma_{p}(G) $ is the minimum cardinality of a perfect dominating set of $ G $ \cite{Tha}.\\

 The weight of a function $ f $ defined on the vertex set $ V $ of a graph $ G $, $ f(V)  $ is $ \sum_{ u\in V}f(u) $. A map $ f:V(G)\rightarrow\{0,1,2\} $ is a Roman dominating function for a graph $G$ if for every vertex $v$ with $f(v)=0$, there exists at least one vertex $ u \in N(v) $ such that $ f(u)=2$. The minimum weight of a Roman dominating function on $G$ is called the Roman domination number of $G$, $ \gamma_R(G) $  \cite{EJC}.\\

  An Italian dominating function, of a graph $G$ is a function $ f: V(G) \rightarrow \{0,1,2\} $ satisfying the condition that for every $ v\in V(G) $ with $ f(v) = 0, \sum_{ u\in N(v)} f(u) \geq 2, $ i.e., either $v$ is adjacent to a vertex $u$ with $ f(u) = 2 $ or to at least two vertices $x$ and $y$ with $ f(x) = f(y) = 1. $ The Italian domination number of $ G $, $ \gamma_I(G) $ is the minimum weight of an Italian dominating function on $ G $ \cite{Mustha}.\\

   A function $ f: V(G)\rightarrow \{0,1,2\} $ is a perfect Italian dominating function (abbreviated as PID-function) on $G$ if for every vertex $ v \in V(G) $ with $ f(v) = 0 $, $ \sum_{ u\in N(v)}f(u)=2 $. The perfect Italian domination number of $G$, $ \gamma^{p}_{I}(G), $ is the minimum weight of a PID-function of $G$. A PID-function of $G$ with weight $ \gamma^{p}_{I}(G) $ is called a $ \gamma^{p}_{I}(G)$-function of $G$ \cite{Twh}.\\

    We also denote a function $ f:V(G)\rightarrow\{0,1,2\} $ as $ f= (V_0^f,V_1^f,V_2^f) $ or simply $ (V_0,V_1,V_2) $, where $ V_i $ is the set of all vertices which are assigned the value $ i $ for $ i=0,1,2 $. For any subgraph $ H $ of $ G $, the sum of the weights of the vertices of $ H $ is denoted by $ f(H) $. i.e., $f(H)= \sum_{u \in V(H)}f(u)$. In \cite{Lau} the authors characterize the graphs $ G $ with $ \gamma_I^p(G) $ equal to 2 and 3 and determined the exact value of the parameter for several simple structured graphs. It is also proved that it is NP-complete to decide whether a given bipartite graph admits a perfect Italian dominating function of weight $ k $. The perfect Italian domination number of Sieri\'{n}ski graphs and generalized Sierpi\'{n}ski graphs are studied in \cite{Jis}  and \cite{Ji} respectively.\\

  For disjoint graphs $ G $ and $ H $, the join $ G+H $ has vertex set $ V(G)\cup V(H) $ and edge set $ E(G) \cup E(H) \cup \{uv: u \in V(G)\ and\ v\in V(H)\}$ \cite{Tha}. The Cartesian product   of two graphs $G$ and $H$, $G \square H$ has vertex $ V(G) \times V(H) $ and two vertices $ (u_1,v_1) $ and $(u_2,v_2)$ are adjacent if either $u_1=u_2$ and $v_1v_2 \in E(H)$ or $v_1=v_2$ and $u_1u_2 \in E(G)$ \cite{Jacob}. It is a simple observation that $G \square H$ can be partitioned as $ \lvert V(H)\lvert $ copies of $G$ and $\lvert V(G) \lvert$ copies of $H$.\\

   The Mycielskian  of a graph $ G $, $ M(G) $ is the graph with vertex set $ V(G) \cup V'(G) \cup \{w\} $ where $ V'(G)= \{u_i: v_i \in V(G)\} $ and edge set $ E(G) \cup \{v_iu_j : v_iv_j \in E(G) \} \cup \{ wu_i : u_i \in V'(G) \}. $ The double Roman domination number and the Italian domination number of the Mycielskian of a graph have been studied in \cite{Anu1} and \cite{Jism}. \\

The following observations are simple.\\

\ni\textbf{Observation 1.} For a graph with no edge and $n$ vertices, $ \gamma^{p}_{I}(G) = n.$\\

\ni\textbf{Observation 2.} For any complete bipartite graph $ K_{p,q}, $
 \[
\gamma^{p}_{I}(K_{p,q}) =
\begin{cases}
4,\  p,q\geq3,\\
2,\  otherwise.
\end{cases}
\]

\ni\textbf{Observation 3.} For complete graph $ K_m, $ $ \gamma^{p}_{I}(K_m) = 2. $\\

\ni\textbf{Observation 4.} For every graph $G$, $ \gamma(G) \leq \gamma_{I}(G) \leq \gamma^{p}_{I}(G).$\\

\ni\textbf{Observation 5.} Let $G$ be a graph. $ \gamma^{p}_{I}(G) = 2 $ if and only if $ G = H_1 \vee H_2$ where $ H_1 = K_1, K_2\  or\  2K_1 $.	
\begin{proof}
If $\gamma^{p}_{I}(G)= 2,$ in a PID-function of $G$, either a vertex $v$ is assigned the value 2 and all the remaining vertices are adjacent to $v$ or two vertices $v$ and $w$ are assigned the value 1 and all the remaining vertices are adjacent to both $v$ and $w$. The adjacency between $v$ and $w$ is optional. Therefore, $G$ is $ K_1\vee H_2,\ K_2\vee H_2\ or\ 2K_1\vee H_2. $ The converse is a simple observation.
\end{proof}
All notations and terminology not mentioned here are from \cite{Bal}.\\

\section{Cartesian Product}
 In this section, we have obtained an upper bound for the Cartesian product of two graphs in terms of the original graph. Exact values for some special classes are also obtained.

 \begin{thm} \label{thm8}
 	For any graphs G and H
 	\begin{center}
 		$ \gamma^{p}_{I}(G \square H) \leq min\{|V(H)|\gamma^{p}_{I}(G),|V(G)|\gamma^{p}_{I}(H)\}. $
 	\end{center}
 \end{thm}
 \begin{proof}
 	Let $g$ be $ \gamma^{p}_{I}  $-function of $G$. Let $ f:V(G)\times V(H)\rightarrow\{0,1,2\} $ be $ \gamma^{p}_{I} $-function of $ G \square H $ defined by $f(u,v)=g(u)$, for every $ u\in V(G)\ $ and $ \ v\in V(H) $. Then a vertex $(u,v)$ has weight zero, then it has neighbors with weight exactly two and all other vertices which are adjacent to $(u,v)$ has weight zero. Therefore, $f$ is a $ \gamma^{p}_{I}  $-function, and $ \gamma^{p}_{I}(G \square H) \leq |V(H)|\gamma^{p}_{I}(G)$.  Using the same arguments we can prove that $ \gamma^{p}_{I}(G \square H) \leq |V(G)|\gamma^{p}_{I}(H).$ Therefore,
 	\begin{center}
 		$ \gamma^{p}_{I}(G \square H) \leq min\{|V(H)|\gamma^{p}_{I}(G),|V(G)|\gamma^{p}_{I}(H)\}. $
 	\end{center}
 \end{proof}
 There are examples of pairs of graphs for which equality and strict inequality of the above theorem are attained. For instance, let $ G = P_4 $ and $ H=P_2 $. Then $ \gamma^{p}_{I}(G \square H)\ = 4 < 6 = min\{|V(H)|\gamma^{p}_{I}(G),|V(G)|\gamma^{p}_{I}(H)\} $ and  let $ G = K_{1,3} $ and  $ H = P_3.  $ Then $ \gamma^{p}_{I}(G \square H)\ = 6 = min\{|V(H)|\gamma^{p}_{I}(G),|V(G)|\gamma^{p}_{I}(H)\}.$\\

The following theorem proved in \cite{Aliz} is used in the proof of Theorem \ref{thm6}
\begin{thm}(\cite{Aliz}.)
	$  \gamma_{I}(P_2\square P_n)=n. $
\end{thm}
\begin{thm} \label{thm6}
		\[
	 \gamma^{p}_{I}(P_2 \square P_n) = \begin{cases} n+1;\ if\ n\  =\ 1,3,5\\
	 n;\ otherwise.
	\end{cases}
	\]
\end{thm}

\begin{proof}

Let $ f = (V_0,V_1,V_2) $ be the $ \gamma^{p}_{I}  $ function of $ P_2 \square P_n. $ Let $ u_1, u_2, u_3,...,u_n $ be the vertices of the first copy of $ P_n $ and $ v_1,v_2, v_3,...,v_n $ be the vertices of the second copy of $ P_n.$
We know that $ \gamma_{I}(P_2 \square P_n) = n  $ and by observation 4, $ \gamma_{I}(G) \leq \gamma^{p}_{I}(G). $ Therefore, $ \gamma^{p}_{I}(P_2 \square P_n)\geq n. $\\

\ni When $n=2$, $ P_2 \square P_n $ is $ C_4 $ and $ \gamma^{p}_{I} (C_4)= 2 $.\\

\ni When $n = 3$, define $f$ as follows.
\[
f(u)=
\begin{cases}
2;\  u=v_3,\\
1;\ u=u_1,v_2,\\
0;\ otherwise.
\end{cases}
\]
  Then  $ \gamma^{p}_{I}(P_2 \square P_n) = 4. $ \\

\ni When $ n = 4$, define $f$ as follows.
\[
f(u) =
\begin{cases}
1;\ u=u_2,u_3,v_1,v_4,\\
0;\ otherwise.
\end{cases}
\]

Then  $ \gamma^{p}_{I}(P_2 \square P_n) = 4. $ \\

\ni  When $n = 5$, define $f$ as follows.
  \[
  f(u)=
  \begin{cases}
  2;\  u=u_5\\
  1;\ u=u_1,u_4,v_2,v_3,\\
  0;\ otherwise.
  \end{cases}
  \]
  Then $ \gamma^{p}_{I}(P_2 \square P_n) = 6.$
 \\

\ni When $ n = 6$, define $f$ as follows.
\[
f(u) =
\begin{cases}
1;\ u=u_2,u_3,u_6,v_1,v_4,v_5,\\
0;\ otherwise.
\end{cases}
\]

Then  $ \gamma^{p}_{I}(P_2 \square P_n) = 6. $ \\

When $n\geq7$, and $n$ is odd, define $f$ as follows.
\[
f(u) =
\begin{cases}
2;\ u=u_j,\ j\equiv 4(mod6)\\
1;\  u=u_j,\ j\equiv 1(mod6),\\\hspace{.6cm}u=v_j,\ j\equiv 0(mod2);\\
0;\  otherwise.
\end{cases}
\]
When $n$ is even, define $f$ as follows
\[
f(u) =
\begin{cases}
1;\  u=u_j,\ j\equiv 0,1(mod4),\\\hspace{.6cm}u=v_j,\ j\equiv 2,3 (mod4);\\
0;\  otherwise.
\end{cases}
\]
Clearly, in each case, $f$ is a $ \gamma_I^p $-function and $ f(V) = n. $ Hence the theorem.

\end{proof}

\begin{thm} \label{thm7}
If $m$ and $n$ are positive integers then
\[ \gamma^{p}_{I}(K_m \square K_n) =
\begin{cases}
n;\ if\  m=n\\
min\{2m,2n\};\ otherwise.
\end{cases}\]
\end{thm}
\begin{proof}
Let $ f = (V_0,V_1,V_2) $ be the $ \gamma^{p}_{I}  $-function of $ K_m \square K_n.$ As we have already mentioned in the introduction $ K_m \square K_n $ can be viewed as $m$ rows of $ K_n $ and $n$ columns of $ K_m.$ Let $ u_{i,j}, $ $i=1,2,..m$ and $j=1,2,...n$ be the vertices of $ K_m \square K_n.$\\

\ni{\bf Case 1:} $ m=n. $\\

Define $f$ as follows.
\[
f(u_{ij})=
\begin{cases}
1;\  i=j,\\
0;\ otherwise.
\end{cases}
\]
Then $ \gamma^{p}_{I}(K_n \square K_n) \leq n. $\\
\ni{\bf Claim:} Exactly one vertex in each copy of $ K_n $ has weight 1.\\
 If possible assume that there exists a copy of $ K_n $ in which all vertices have weight 0. Then these vertices are dominated by vertices from corresponding columns. Then each column should have weight 2, i.e., $f(V)=2n>n$.\\ If possible assume that there exist a copy of $ K_n $ which has weight at least 2. Then either there is a vertex with weight 2 or two vertices with weight 1 each in that row.\\

\ni {\bf Case (a):} Let $ u_{ij} $ and $ u_{ik} $ be the two vertices with weight 1.\\

 Then in the $ i^{th} $ row either all  other vertices  have  weight 1 or all other vertices have weight 0. If all other vertices are assigned zero then vertices in the corresponding column is zero. In order to dominate these vertices  we have to assign weight 2 in each row. Then $ f(V)=2n >n.$ If all other vertices are assigned weight 1, then to dominate any vertex with weight 0 in any other row we have to assign a vertex with weight 1 in each row. Then $ f(V)= 2n-1>n. $ \\

\ni{\bf Case(b):} Let $ u_{ij} $ be a vertex in $ i^{th} $ row that has weight 2.\\

 Similar to case(a), we can prove that, in this case also $ f(V)=2n >n.$\\ Therefore, weight of each row is one and hence, $ \gamma^{p}_{I}(K_n \square K_n) = n. $\\

\ni{\bf Case 2:} $ m\neq n. $\\

Without loss of generality, let $m<n$. Define $f$ as follows.
\[
f(u_{ij})=
\begin{cases}
2;\  i=1,2,...m\  and\ j=1,\\
0;\ otherwise.
\end{cases}
\]
Then $ \gamma^{p}_{I}(K_m \square K_n)\leq 2m = min\{2m,2n\}. $\\
If for every $ \gamma_I^p $-function $f$, $ \sum_{j=1}^n f(u_{ij})=2,$ for each $i$, then $ \gamma^{p}_{I}(K_m \square K_n) = 2m. $ Therefore, assume that there exists a $ \gamma_I^p $-function $f$ such that $ \sum_{j=1}^n f(u_{ij})<2$ for some $i=k$. Therefore, $ \sum_{j=1}^{n}f(u_{kj}) = 0 $ or 1. \\

If $ \sum_{j=1}^{n}f(u_{kj}) = 0 $ then to dominate $ u_{kj} $ for $j=1,2,...n$, $ \sum_{i=1}^{m}f(u_{ij}) =2 $ which implies $f(V)= 2n >2m,$ which is a contradiction to the fact that $f$ is a $ \gamma_I^p $-function.\\

 If  $ \sum_{j=1}^{n}f(u_{kj}) = 1 $ then there exists $ l$ such that $ f(u_{kj})=0, $ if $ j\neq l $ and $ f(u_{kl})=1. $ But then to dominate $ u_{kj},\ j\neq l,\  $$ \sum_{i=1}^{m}f(u_{ij}) =1. $ i.e., exactly one vertex in each column has weight 1 and all other vertices have weight 0. But number of rows is less than number of columns. Therefore, there are more than one vertex with weight 1 in at least one row, say $ i=k^{'}. $ But, then
$ \sum_{j=1}^n f(u_{k^{'}j})=2$ or $ n. $ If $ \sum_{j=1}^n f(u_{k^{'}j})=2$ then exactly two vertices in the $ {(k^{'})}^{th} $ row have weight 1 and all others have weight 0. Also, the column containing this 0's must be full of 0's. But this contradicts the fact that $ \sum_{i=1}^{m}f(u_{ij}) =1, $ for all $ j\neq l. $ Therefore, $ \sum_{j=1}^n f(u_{k^{'}j})=n.$ But then $ f(V)= n+m-1 \geq m+1+m-1 $ (since, $ n\geq m+1 $) $=2m.  $ \\

Therefore, if $ m\neq n  $ then $ \gamma^{p}_{I}(K_m \square K_n) = 2m, $ where $ m<n. $

\end{proof}

\section{Realization problems}
\begin{thm}\label{thm5}
	Given any two positive integers $a,\ b$ $ \geq 3, $ there exist a graph G and induced subgraph $H$ of $G$ such that $ \gamma^{p}_{I}(G) = a  $ and $ \gamma^{p}_{I}(H)= b. $
\end{thm}

\begin{proof}
	\ni{\bf Case 1:} $ b\leq a. $\\
	
	Let $ G=P_{2a-1} $ and $ H= P_{2b-1} $. Then $ \gamma^{p}_{I}(G) = \lceil\frac{2(a-1)+1}{2}\rceil = a,  $ and $ \gamma^{p}_{I}(H) = \lceil\frac{2(b-1)+1}{2}\rceil = b.  $\\
	
	\ni{\bf Case 2:} $ b>a $.\\
	
	Let $ v_1,v_2,...v_{2b-1} $ be a path on $2b-1$ vertices. Construct $G$ as follows. Let $u$ and $v$ be two vertices adjacent to $ v_{2a-3},v_{2a-2},...v_{2b-1} $ and let $ v_{2a-4} $ be adjacent to $v$ alone. Clearly, $ \gamma^{p}_{I}(G) = \lceil\frac{2a-5+1}{2}\rceil+2=a.$ Also $ H=P_{2b-1} $ is an induced subgraph and $ \gamma^{p}_{I}(H)=b.$
\end{proof}

\begin{lem}
	For any graph G, $ \gamma_R(G) \leq 2\gamma^{p}_{I}(G) -1.$
\end{lem}

\begin{proof}
Let $f=(V_0^f,V_1^f,V_2^f)$ be a $ \gamma^{p}_{I}  $-function of $G$. Let $ u\in V_1^f. $ Define $ g=(V_0^f,{u},V_1^f\cup V_2^f-{u}). $ Since every vertex in $ V_0^f $ is adjacent to exactly one vertex in $ V_2^f $ or two vertices in $ V_1^f $, in $g$ every vertex in $V_0^g$ will have at least one neighbor with weight 2. Therefore, this assignment gives a Roman dominating function. Now $ g(V) = 2(|V_1^f \cup V_2^f|)-1 \leq 2\gamma_I^p(G)-1. $ Therefore, $ \gamma_R(G) \leq 2\gamma^{p}_{I}(G) -1.$
\end{proof}

If $ \gamma_R(G) = 1$, $G$ is $ K_1, $ and $ \gamma^{p}_{I}(G)= 1 $ and vice versa. Similarly, if $ \gamma_R(G) = 2$ then $G$ has a universal vertex and $\gamma^{p}_{I}(G)= 2,  $ but the converse is not true. If $ \gamma_R(G) = 3$ and $\gamma^{p}_{I}(G)= 2,  $ then let $ G= K_{2,n}. $\\

\begin{thm} \label{thm9}
Given $a,b$ $ \geq 3 $ such that $ a\leq 2b -1, $ then there exists  a graph $G$ such that $ \gamma_R(G) = a\   and\  \gamma^{p}_{I}(G)= b. $
\end{thm}

\begin{proof}

\ni{\bf Case a:}  $ a\geq 3 $,  $ b\geq a+1 $ and $a$ is odd.\\

Consider $ K_p^c \vee P_{b-3}, $ where $p$ is arbitrarily large. Attach a vertex $u$ to every vertex of $ K_p^c \vee P_{b-3} $  and a pendent vertex $v$ to an end vertex of $ P_{b-3}. $ Then  $ \gamma_R(G) = 3$ where $ \gamma_R $-function $f$ can be defined as $f(u)=2,\  f(v)=1$ and $ f(v_i)=0 $ for all other vertices. Also, $\gamma^{p}_{I}(G)= b $ where $ \gamma_I^p $-function $g$ can be defined as $g(u)=2,\  g(v)=1$ and all the vertices of $ P_{b-3} $ has weight 1. If we attach a $ P_3 $ to the vertex $u$ by an edge then $ \gamma_R(G) = 3+2=5  $ and $ \gamma^{p}_{I}(G)=b+2 $. Similarly, by attaching $ P_{3k} $ to the already attached  $ P_3 $, we can get $ \gamma_R(G) = 3+2k  $ and  $ \gamma^{p}_{I}(G)=b+2k, b > a. $\\

\ni {\bf Case b:}  $ a\geq 4 $,  $ b\geq a+1 $ and $a$ is even. \\

Consider $ K_p^c \vee P_{b-4} ,$ where $p$ is arbitrarily large. Attach a vertex $u$ to every vertex of $ K_p^c \vee P_{b-4} $  and attach both the vertices $p$ and $q$ of $ K_2 $ to an end vertex of $ P_{b-4} $. Then $ \gamma_R(G)=4, $ where $ \gamma_R $-function $f$ can be defined as $f(u)=2$ and $f(p)=f(q)=1$. Also,   $ \gamma^{p}_{I}(G)=b$, where $ \gamma_I^p $-function $g$ can be defined as $g(u)=2,\  g(p)=g(q)=1$ and all the vertices of $ P_{b-4} $ has weight 1. Attach a $ P_3 $ with $u$ by an edge then $ \gamma_R(G)= 4+2  $ and $ \gamma^{p}_{I}(G)=b+2 $ . Similarly, by attaching  $ P_{3k} $ as in the previous case  we can get $ \gamma_R(G) = 4+2k  $ and  $ \gamma^{p}_{I}(G)=b+2k, b > a. $\\

\ni{\bf Case c:}  $a=b$ and $a$ is odd.\\

Consider $ P_{2a-1} $. Let $ v_1,v_2,v_3,...,v_{2a-1} $ be the vertices of $ P_{2a-1} $. Let $u$ be a vertex which is attached to $ v_2,v_4,v_6,...v_{2a-2} $ and also $ v_1 $ and $ v_{2a-1}. $ Then $ \gamma_R(G)= a$, where $ \gamma_R $-function $f$ can be defined as $f(u)=2,$  $ f(v_3)=f(v_5)=f(v_7)=...=f(v_{2a-3})=1, $ $f(v)=0$, for all other vertices and $ \gamma^{p}_{I}(G)=a$, where $ \gamma_I^p $-function $g$ can be defined as  $ g(v_1)=g(v_3)=g(v_5)=...=g(v_{2a-1})=1 $, $g(v)=0$, for all other vertices.  In particular, when $a=5$ the graph is given in Figure 1.\\

\begin{figure}[ht!]
	\centering
	
	\includegraphics[width=25mm]{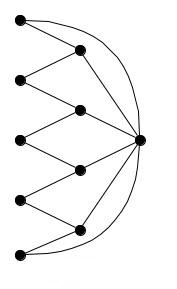}
	\caption{A graph with $a=b=5$ \label{overflow}}
\end{figure}

Similarly, we can construct all graphs with $ \gamma_R(G)=\gamma^{p}_{I}(G)$ when $ \gamma_R(G) $ is odd.
So we have constructed all graphs with $ a\leq b.$\\

\ni{\bf Case d:} $ a>b $.\\

 Let $G$ be the graph constructed as follows. Let $ v_1,v_2,,...,v_b $ be a set of independent vertices. Corresponding to every pair $ (v_i,v_j), i\neq j $ let $ u_{ij} $ be a vertex adjacent to $ v_i $ and $ v_j $ alone. Then $ \gamma^{p}_{I}(G)=b, $ where $f(v_i)= 1,$ for all $i=1,2,3,...,b$ and $ f(u_{ij})=0, $ for all $ i,j\in \{1,2,3,..b\} $ and $ i\neq j $ is a $ \gamma_I^p $-function of $G$. But $ \gamma_R(G)=2b-1, $ where $ g(v_i)=2, $ for $ i=1,2,3,...,b-1 $ and $g(v_b)=1 $ is a $ \gamma_R $-function. In particular, when $a=7$ and $b=4$, the graph is given in Figure 2. \\

\begin{figure}[ht!]
	\centering
	
	\includegraphics[width=35mm]{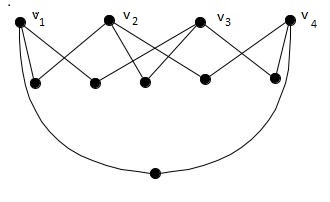}
	\caption{A graph with $a= 7$ and $b=4$ \label{overflow}}
\end{figure}
 Now, if we delete one vertex from the $ ^bC_2 $ vertices, $ \gamma^{p}_{I}(G) $ will not change, whereas $\gamma_R(G)  $ reduces by 1. (Note that $ h(v_k)=2 $ for all $k\in \{1,2,3,...,b\} \setminus\{i\}$ and $ h(v_i)=1 $ is a $ \gamma_R(G) $-function of $G$). Proceeding like this we can reduce $ \gamma_R(G) $ up to $ \gamma_I^p+1.$ Hence the theorem.

\end{proof}

\section{Relation with perfect domination number}
In this section, we study the relationship between perfect Italian domination number and perfect domination number of a graph. The following are simple obsevations.\\

\ni\textbf{Obsevation 6:} $ \gamma_{I}^p(G) \leq 2\gamma_{p}(G)$.\\
\begin{proof}
Let $ P $ be a $ \gamma_{p} $-set of $ G $. Then \\
\[
f(v) =
\begin{cases}
 2;\  if\  v \in P,\\
0;\ otherwise.
\end{cases}
\]
is a PID-function. Therefore, $ \gamma_{I}^p(G) \leq 2\gamma_{p}(G)$.
\end{proof}
\ni \textbf{Obervation 7:} If there exists a $ \gamma_{I}^p $-function of $ G $ such that $ V_1 = \phi$ then $ \gamma_{I}^p(G)= 2\gamma_{p}(G) $.
\begin{proof}
	If there exists a $ \gamma_{I}^p $-function of $ G $ such that $ V_1 = \phi$, then the vertices which are assigned the value 2 forms a $ \gamma_{p} $-set. Therefore, $ \gamma_{p}(G) \leq \frac{1}{2} \gamma_{I}^p(G) $. Hence, $ \gamma_{I}^p(G) \geq 2 \gamma_{p}(G) $ and by obervation 6, $ \gamma_{I}^p(G)= 2\gamma_{p}(G) $.
\end{proof}

We know that the Italian domination number of a graph $ G $ lies between $ \gamma(G) $ and $ 2\gamma(G) $. Here we have proved that $ \gamma_{I}^p(G) \leq 2\gamma_{p}(G)$. It is most expected that $ \gamma_{p}(G) $ serves as a lower bound for $ \gamma_I^p(G) $. But this is not true and $ \gamma_I^p(G) $ can be arbitrarily smaller than $ \gamma_{p}(G) $. the following theorem settles the corresponding realization problem.
\begin{thm}
	Given any two positive integers $ a$, and $ b $ such that $b\leq 2a $ there exists a graph $ G $ such that $ \gamma_{p}(G) =a $ and $ \gamma_{I}^p(G) =b $.
\end{thm}
\begin{proof}
	Let $ a $ and $ b $ be any two positive integers such that $b\leq 2a $.\\
	
	\ni\textbf{Case 1:} $a \leq b \leq 2a-1$.\\
	
	Consider k copies of $ P_5 $, say $ v_{i1}v_{i2}v_{i3}v_{i4}v_{i5} $ for $ i=1,2,...k $, where $ v_{i1}=v_{j1} $, for all $ i,j \in \{1,2,...,k\} $. Then $ \gamma_{I}^p(G)= 2k+1 $, where $ \gamma_{I}^p $-function $ f $ can be defined as $ f(v_{i1})=f(v_{i3})= f(v_{i5})=1 $, for all $ i=1,2,...,k $ and 0, otherwise. Also $ \gamma_{p}(G) = k+1$, where the perfect dominating set consist of the vertices $ v_{i1} $ and $ v_{i4} $, $ i=1,2,...,k $. If we extend the path $ v_{11}v_{12}v_{13}v_{14}v_{15} $ to a path of length $ 2l+5 $, then $ \gamma_{I}^p(G)=2k+1+l $ and $ \gamma_{p}(G)=k+1+l $. Let $ k=b-a $ and $ l=2a-b-1 $, so that $ \gamma_{p}(G)=a $ and $ \gamma_I^p=b $.\\
	
	\ni\textbf{Case 2} $ b=2a $.\\
	
	Let $ G $ be the graph $ P_a:v_1v_2,...,v_a $, with atleast two pendent vertices attached to every $ v_i,\ i=1,2,...,a $. Then $ \gamma_I^p(G)=2a $ and $ \gamma_{p}(G)=a $.\\
	
	\ni\textbf{Case 3:} $ a > b $.\\
	
	\ni\textbf{Subcase (a):} $ b-a $ is even.\\
	
	Let $ G=K_2^c + k K_2  $. Then $ \gamma_I^p(G)=2 $, where vertices of $ K_2^c $ is assigned the value 1 and others 0, is the $ \gamma_I^p $-function of $ G $. But the $ \gamma_{p} $-set contains all the vertices of the graph and hence $ \gamma_{p}(G) = 2k+2$. By attaching a path of length $ 2l $ to one of the vertices of $ K_2^c $, as in case 1, we get $ \gamma_I^p(G)=2+l $ and $ \gamma_{p}(G)= 2k+2+l $. Let $k= \frac{a-b}{2}$ and $ l=b-2 $ so that $ \gamma_{p}(G)=a $ and $ \gamma_I^p=b $.\\
	
	\ni\textbf{Subcase (b):} $ b-a $ is odd.\\
	
	Let $ G= K_2^c + (K_3 \cup kK_2) $. Then as in the previous case, $ \gamma_I^p(G)=2 $ and $\gamma_{p}(G) = 2k+5$. By attaching a path of length $ 2l $ to one of the vertices of $ K_2^c $, we get $ \gamma_I^p(G)=2+l $ and $ \gamma_{p}(G)= 2k+5+l $. Let $ k=\frac{a-b-3}{2} $ and $ l=b-2 $ so that $ \gamma_{p}(G)=a $ and $ \gamma_I^p=b $.\\
	\end{proof}
\section{Mycielskian of a graph}
In this section, we study the relationship between the perfect Italian domination number of Mycielskian of a graph and the perfect domination number of the graph.
\begin{thm}
	For a connected graph $ G $, $ \gamma_I^p(M(G)) \leq 2\gamma_{p}(G)+1 $.
\end{thm}
\begin{proof}
	Let $ P $ be a $ \gamma_{p} $-set of $ G $. Let $ P' = \{u_i,v_i: v_i \in P \} \cup \{ w\} $. Define a PID-function as follows.
	\[
	f(v)=
	\begin{cases}
	1;\ if\ v \in\ P',\\
	0;\ otherwise.
	\end{cases}
	\]
	Then $ f(M(G))= 2\gamma_{p}(G)+1 $. Therefore, $ \gamma_I^p(M(G)) \leq 2\gamma_{p}(G)+1 $.
\end{proof}
Although, many graph classes satisfy equality, it may be noted that there are infinite number of families of graph which satisfy strict inequality. If we conside the graph $ G= K_2^c \bigvee kK_2  $, then $ \gamma_I^p(M(G))=6 $ and $ \gamma_{p}(G)=2k+2 $, so that the difference can be arbitrarily large. An illustration where $ k=2 $ is given in Figure 3, in which $ f(v_3)=f(v_4)=f(u_1)=f(u_2)=1 $ and $ f(w)=2 $ gives a $ \gamma_{I}^p $-function of $ M(G) $.
\begin{figure}[h] \label{Fig1}
	\centering
	\includegraphics[width=8.5cm]{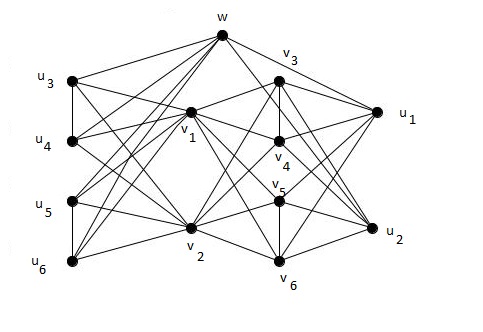}\\
	\caption{$M(K_2^c\vee 2K_2)$.\label{overflow} }	
\end{figure}
\begin{thm} \label{thm10}
	If $ G $ has a $ \gamma_I^p $-function such that $ V_1 = \phi$, then  $ \gamma_I^p(M(G))=\gamma_I^p(G)+1 = 2\gamma_{p}(G)+1$.
\end{thm}
\begin{proof}
	Assume that $ G $ has a $ \gamma_I^p $-function $ f $ such that $ V_1^f= \phi $. We can define a perfect Italian dominating function $ g: V(M(G))\rightarrow \{0,1,2\} $ as follows.
	\[
	g(v)=
	\begin{cases}
	1;\  for\ v=v_i\ and\  u_i\ such\ that\ f(v_i)=2\ and \  v=w,\\
	0;\ otherwise.
	\end{cases}
	\]
	Therefore, $ \gamma_I^p(M(G))\leq\gamma_I^p(G)+1  $.\\
	
	To prove the reverse inequality, let $ f $ be a PID-function of $ M(G) $.\\
	
	\ni\textbf{Case1:} $ \sum_{i=1}^{n}f(u_i) \neq 2$.\\
	
	Define $ g:V(G)\rightarrow \{0,1,2\} $ as follows.
	\[
	g(v_i)=
	\begin{cases}
	2;\ if\ f(v_i)+f(u_i)=2,\\
	1;\ if\ f(v_i)+f(u_i)=1,\\
	0;\ otherwise.
	\end{cases}
	\]
	Note that $ f(v_i)+f(u_i) \leq 2 $, for every $ i=1,2,...,n $.
	Then $ g $ is a PID-function of $ G $, since $ N_G(v_i) = N_{M(G)}(u_i) \bigcap V(G)$.\\
	
	\ni\textbf{Case2:} $ \sum_{i=1}^{n}f(u_i) = 2 $.\\
	
	In this case, either there exists one vertex $ u_i $ with $ f(u_i)=2 $ or there exist two vertices $ u_i $ and $ u_j $ with $ f(u_i)=f(u_j)=1 $.\\
	
	\ni\textbf{Subcase(a):} There exists one vertex $ u_i $ with $ f(u_i)=2 $.\\
	
	Without loss of generality let $ f(u_1) =2$ and $f(u_i)=0 $, for all $ i=2,3,...,n $. If possible assume that there exists a $ v_i \in N(u_1)$ such that $ f(v_i)=0 $. But, we have $ f(u_i) =0$ which implies that $ \sum f(N(u_i))=2 $ and hence, $ \sum f(N(v_i))= \sum f(N(u_i)) + f(u_1)-f(w)=2$. This implies, $ f(w)=2 $. Therfore, $ \sum_{ x\in N(u_i)-\{w\}}f(x)=0, $ for all $ i=2,3,...,n $. This implies $ f(x)=0 $ for $ x \in N(v_i) $, for all $ i=2,3,...,n $. But, then $ v_i's $ are perfect Italian dominated by $ u_i $. This means $ v_1 $ is a universal vertex of $ G $ and also $ f(M(G)) \geq 4 $. But, $g(v_1) =g(u_1) = g(w)=1$ is a PID-function of $M(G)$  with weight 3, which is a contradiction to the fact that $ f $ is a $ \gamma_{I}^p $- function of $ M(G) $. Therefore, none of the vertices in $ V(G) $ is perfect Italian dominated by $ u_1 $. Therefore, $ f $ restricted to $ G $ is a PID-function of $ G $ and $ f(G) \leq f(M(G))-2$. Therefore, $ \gamma_I^p(G) \leq \gamma_I^p(M(G))-2 $ so that, $ \gamma_I^p(M(G))\geq\gamma_I^p(G)+2  $, which is a contradiction to the fact that  $ \gamma_I^p(M(G))\leq\gamma_I^p(G)+1  $. Therefore, such a case dose not exist.\\
	
	\ni\textbf{Subcase(b):} There exist two vertices $ u_i $ and $ u_j $ with $ f(u_i)=f(u_j)=1 $.\\
	
	Without loss of generality, let $ f(u_1) =f(u_2) =1$ and $f(u_i)=0 $, for all $ i=3,4,...,n $. As in the above case, there dose not exist $ v_i \in N(u_1) \bigcup N(u_2) $, $ i \neq 1,2 $ such that $ f(v_i)=0 $. If $ v_1,\ v_2 \notin N(u_1) \bigcup N(u_2)$, then again $ f/G $ is a PID-function of $ G $ and hence $ \gamma_I^p(G) \leq \gamma_I^p(M(G))-2 $, which is not possible. Therefore, assume that $ v_1,\ v_2\ \in \ N(u_1)\bigcup\ N(u_2) $. i.e., $ v_1 \in N(u_2) $ and $ v_2 \in N(u_1) $. If $ f(v_1)=f(v_2)=0 $ then define $ g:V(G)\rightarrow \{0,1,2\} $ as follows.\\
	\[
	g(v_i)=
	\begin{cases}
	f(v_i);\ i\neq\ 1,2,\\
	1;\ i=1,\\
	0;\ i=2.
	\end{cases}
	\]
	If $ f(v_1)=0 $ and $ f(v_2)\neq 0 $ then define $ g:V(G)\rightarrow \{0,1,2\} $ as follows.\\
	\[
	g(v_i)=
	\begin{cases}
	f(v_i);\ i\neq\ 2,\\
	f(v_2)+1; i=2.
	\end{cases}
	\]
	The case $ f(v_1) \neq 0 $ and $ f(v_2)=0 $ can be delt similarly. If both $ f(v_1) $ and $ f(v_2) $ are non-zero then $ f/G $ is a PID-function, which again leads to a contradiction. In all the cases $ \gamma_I^p(G) \leq \gamma_I^p(M(G))-1 $. Therefore, $\gamma_I^p(M(G))=\gamma_I^p(G)+1$.\\
	By Observation 7, we know that $ \gamma_{I}^p(G)= 2\gamma_{p}(G) $. Therefore, $ \gamma_I^p(M(G))= 2\gamma_{p}(G) +1$.
\end{proof}
\begin{cor}
	For any graph $ G $ with a universal vertex, $ \gamma_I^p(M(G)))=3 $.
\end{cor}
\section{Conclusion and Open Problems}
In this paper, we have already given some examples of graphs which satisfy $ \gamma_{I}^p(G)=2\gamma_{p}(G) $. Let $ G $ be a graph with $ n $ vertices and let  $ \Im = \{H_i: i=1,2,...,n\}$ be a family of n graphs (not necessarily non-isomorphic). We define the corona of $ G $ with $ \Im $, $ G \odot \Im $ as the graph with vertex set $ V(G) \cup V(H_i) $, $ i=1,2,...,n $ and edge set $ E(G) \cup E(H_i) \cup \{v_iu: u \in V(H_i),$ for\ all $i=1,2,...,n \} $. When $ H_i = H $, for all $ i $, $ G \odot \Im $ reduces to the usual corona of $ G $ and $ H $, $ G\odot H $. $ G \odot \Im $ satisfies, $ \gamma_{I}^p(G\odot \Im)=2\gamma_{p}(G\odot \Im) $, if $ |V(H_i)| > 1 $ for all $ i=1,2,...n $. Any supergraph of the above graph obtained by adding edges between $ H_i's $, to some extend, also satisfy the above equality.

\begin{figure}[h] \label{Fig4}
	\centering
	\includegraphics[width=7.5cm]{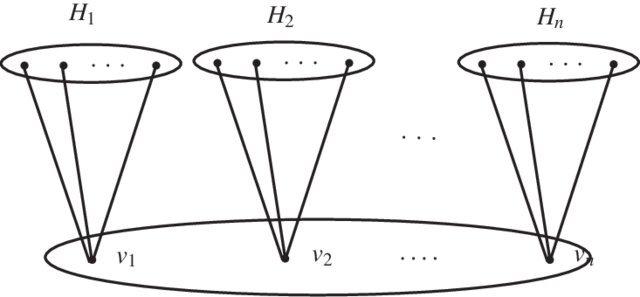}\\
	\caption{Structure of $ G\odot \Im $.\label{overflow} }	
\end{figure}
Though we have infinitely many graphs which satisfy this equality, the charecterization problem is still open.\\

\ni\textbf{ Problem 1:} Charecterize graphs for which $ \gamma_{I}^p(G)=2\gamma_{p}(G). $\\

In Theorem \ref{thm10} we have proved that, if $ G $ has a $ \gamma_I^p $-function such that $ V_1 = \phi$, then  $ \gamma_I^p(M(G))=\gamma_I^p(G)+1 = 2\gamma_{p}(G)+1$. Let $ G $ be the graph $ P_3 \odot K_1$. Here, $ \gamma_{p}(G)=3 $, $ \gamma_I^p(G)=4 $ and $ \gamma_I^p(M(G)) =\ 7\ = 2\gamma_{p}(G) +1$. But there dose not exist a  $ \gamma_I^p $-function of $ G $ in which $ V_1 = \phi $. Therefore, the converse of the theorem is not true for the equality $ \gamma_I^p(M(G))=2\gamma_{p}(G)+1 $. But we strongly belive that the converse of the Theorem \ref{thm10} is true for $ \gamma_I^p(M(G))= \gamma_I^p(G)+1 $.  So we have the following open problems.\\

\ni\textbf{Problem 2:} Charecterize graphs for which   $ \gamma_I^p(M(G)) = 2\gamma_{p}(G)+1$. \\

\ni\textbf{Problem 3:}  Prove the converse of Theorem 5.2 for the equality $ \gamma_I^p(M(G))= \gamma_I^p(G)+1 $. ie; if $ \gamma_I^p(M(G))= \gamma_I^p(G)+1 $, then there exists a $ \gamma_{I}^p $-function of $ G $  for which $ V_1 = \phi $.\\

We know that, if there exists a  $\gamma_I^p  $-function of $ G $ such that $ V_1 = \phi $ then $ \gamma_{I}^p(G)=2\gamma_{p}(G)$. Therefore, if we can prove Problem 3 and if Class A and Class B are the classes of graphs which satisfies Problem 1 and 2 respectively, then the intersection of Class A and Class B is precisely the collection of graphs for which there exists a $\gamma_I^p$-function such that $V_1 = \phi$.\\

{}

\end{document}